%% file: root.tex
\documentclass[letterpaper, 10 pt, conference]{ieeeconf}
\IEEEoverridecommandlockouts
\usepackage{graphics}
\usepackage{epsfig}
\usepackage{times}
\usepackage{amsmath}
\usepackage{amssymb}
\usepackage[normalem]{ulem}
\usepackage{xcolor}
\usepackage{url}
\usepackage{array}
\usepackage{tikz}
\usepackage{mathtools}
\usepackage[noadjust]{cite}

\newtheorem{teo}{Theorem}[section]

\newtheorem{lem}[teo]{Lemma}
\newtheorem{prop}[teo]{Proposition}
\newtheorem{definition}[teo]{Definition}
\newtheorem{ex}[teo]{Example}
\newtheorem{rmk}[teo]{Remark}

\def\R{\mathbb R}
\def\to{\rightarrow}

\DeclarePairedDelimiter{\norm}{\lVert}{\rVert}
\DeclarePairedDelimiter{\abs}{\lvert}{\rvert}

\DeclareMathOperator{\Real}{Re}
\DeclareMathOperator{\Imag}{Im}
\DeclareMathOperator{\diff}{d\!}
\DeclareMathOperator*{\esssup}{ess\,sup}

\newcommand{\ros}{\textrm{ros}}
\newcommand{\bls}{\textrm{bls}}
\newcommand{\full}{\textrm{full}}

\title{\LARGE \bf Frequency domain approach for the stability analysis of a fast hyperbolic PDE coupled with a slow ODE}

\author{Gonzalo Arias$^1$, Swann Marx$^2$, and Guilherme Mazanti$^3$%
\thanks{This work has been partially supported by ANID Millennium Science Initiative Program trough Millennium Nucleus for Applied Control and Inverse Problems NCN19-161 and STIC-Amsud C-CAIT.}%
\thanks{$^1$Gonzalo Arias is with Facultad de Matem\'aticas, Pontificia Universidad Católica de Chile, Avda. Vicu\~na Mackenna 4860, Macul, Santiago, Chile \texttt{\small ngonzaloandres@uc.cl}}%
\thanks{$^2$Swann Marx is with LS2N, \'Ecole Centrale de Nantes \& CNRS UMR 6004, F-44000 Nantes, France \texttt{\small swann.marx@ls2n.fr}}%
\thanks{$^3$Guilherme Mazanti is with Universit\'e Paris-Saclay, CNRS, CentraleSup\'elec, Inria, Laboratoire des signaux et syst\`emes, 91190, Gif-sur-Yvette, France \texttt{\small guilherme.mazanti@inria.fr}}%
}

\begin{document}

\maketitle

\thispagestyle{empty}
\pagestyle{empty}

\begin{abstract}
This paper deals with the exponential stability of systems made of a hyperbolic PDE coupled with an ODE with different time scales, the dynamics of the PDE being much faster than that of the ODE. Such a difference of time scales is modeled though a small parameter $\varepsilon$ multiplying the time derivative in the PDE, and our stability analysis relies on the singular perturbation method. More precisely, we define two subsystems: a reduced order system, representing the dynamics of the full system in the limit $\varepsilon = 0$, and a boundary-layer system, which represents the dynamics of the PDE in the fast time scale. Our main result shows that, if both the reduced order and the boundary-layer systems are exponentially stable, then the full system is also exponentially stable for $\varepsilon$ small enough, and our strategy is based on a spectral analysis of the systems under consideration. Our main result improves a previous result in the literature, which was proved using a Lyapunov approach and required a stronger assumption on the boundary-layer system to obtain the same conclusion.
\end{abstract}


\section{INTRODUCTION}
\label{intro}

Systems defined by partial differential equations (PDEs) coupled with ordinary differential equation (ODEs) appear in many applications, either through the modeling of coupled physical phenomena of different natures, one described by a PDE and the other by an ODE, such as the heated gas flow model from \cite{tang2017stability}, or through the boundary control systems described by PDEs using a dynamic controller, such as the control of a network of liquid fluids through its nodes using proportional-integral controllers from \cite{bastin2015stability}. Such a class of coupled systems has attracted much research effort in recent years and its analysis and control is an active subject \cite{AhmedAli2021Sampled, Auriol2022Advances, bastin2015stability, cerpa2020singular, Irscheid2023Output, tang2017stability, TerrandJeanne2020Regulation}.

In applications, the constituents of a coupled system may model different physical phenomena taking place in different time scales. This is the case, for instance, of electric motors, in which the electrical time scale is typically much faster than the mechanical one (see, e.g., \cite[Example~11.1]{khalil2002nonlinear}). A natural question in these situations, which is the basis of the singular perturbation theory, is whether one can approximate the fastest time scale by an instantaneous process. For linear finite-dimensional systems, its answer is given by Tikhonov's theorem, which states roughly that such an approximation is valid as soon as the dynamics of the fastest time scale is stable \cite{Kokotovic1999Singular}, and the nonlinear theory for finite-dimensional systems is also well-established \cite{khalil2002nonlinear}. Broadly speaking, the main idea of singular perturbation is to decouple the full system into two approximated subsystems under the assumption that the fast dynamics is fast enough. The approximated slow system is called the \emph{reduced order system} and the fast one, after a suitable time rescaling, is called the \emph{boundary-layer system}. Singular perturbation for infinite-dimensional systems is a more delicate topic which has attracted much research effort in recent years \cite{cerpa2020singular,cerpa-prieur2017,marx2022singular,tang2015tikhonov,tang2017stability,Tang2019Singular}, both for coupled PDEs and for PDEs coupled with ODEs, and several examples show that a straightforward generalization of the finite-dimensional theory is not possible \cite{cerpa-prieur2017,tang2015tikhonov,tang2017stability}.
 
The purpose of this paper is to study the stability properties of a system composed by a transport equation coupled through its boundary data with an ODE, when the dynamics of the transport equation is much faster than that of the ODE. More precisely, we consider here the system
\begin{equation}
\label{transport-ode}
\left\{
\begin{aligned}
    & \dot z(t) = A z(t) + B y(1,t), & & t > 0,  \\
    & \varepsilon y_t(x, t) + \Lambda y_{x}(x, t) = 0,& & x \in (0, 1),\, t > 0, \\
    & y(0,t) = G_1 y(1,t)+G_2 z(t),& & t > 0,\\
    & z(0)=z_0\\
    & y(x,0) = y_0(x),& & x\in (0,1),
\end{aligned}
\right.
\end{equation}
where $z\colon [0,\infty) \to \mathbb{R}^n$ is the state of the ODE, $y\colon [0,1] \times [0,\infty) \to \mathbb{R}^m$ is the state of the PDE, $y_x$ and $y_t$ denote the partial derivatives of $y$ with respect to its first and second variables, respectively, $\Lambda$ is a diagonal matrix in $\mathbb{R}^{m \times m}$ with positive diagonal entries, $A, B, G_1, G_2$ are matrices with appropriate dimensions, $I_m - G_1$ is invertible, and $\varepsilon > 0$. The parameter $\varepsilon$ is supposed to be small, meaning that the dynamics of the transport PDE is faster than that of the ODE.

The analysis of \eqref{transport-ode} at the light of the singular perturbation method was previously carried out in \cite{tang2017stability} through the use of a suitable Lyapunov functional, and it was shown that, if the reduced order system associated with \eqref{transport-ode} is exponentially stable and the matrices $\Lambda$ and $G_1$ satisfy a suitable matrix inequality, then the corresponding boundary-layer system is also exponentially stable in $L^2(0, 1; \mathbb R^m)$, and \eqref{transport-ode} is exponentially stable in $\mathbb R^n \times L^2(0, 1; \mathbb R^m)$ if $\varepsilon$ is small enough (see Theorem~\ref{thm:tang2017} below for a precise statement). Here, we rely instead in a spectral approach, which allows us to obtain a stronger result than that of \cite{tang2017stability}.

The sequel of this paper is organized as follows. Section~\ref{sec:Prelim} provides a description of the singular perturbation method applied to \eqref{transport-ode}, describing the corresponding reduced order and boundary-layer systems, and addresses well-posedness issues and the definitions of exponential stability used here. The statement of our main result and its proof are the subject of Section~\ref{sec:MainRes}. Section~\ref{sec:Comparison} compares our result to that of \cite{tang2017stability}, and a numerical illustration of our main result is provided in Section~\ref{sec:Numerical}. Finally, Section~\ref{sec:Concl} concludes the paper.

\smallskip

\noindent\emph{Notation.} Given a matrix $M$ in $\R^{n\times n}$, $M^{-1}$, $M^T$, $e^M$, and $\rho(M)$ denote, respectively, the inverse (when it exists), the transpose, the matrix exponential, and the spectral radius of $M$ (i.e., the largest absolute value of the eigenvalues of $M$). The identity matrix in $\mathbb R^{n \times n}$ is denoted by $I_n$. For $p \in [1, \infty]$, we use $\abs{\cdot}_p$ to denote the $p$ norm in $\mathbb R^n$ as well as the induced matrix norm in $\mathbb R^{n \times m}$, and the index $p$ is omitted from this notation when it is clear from the context or when the choice of the norm is not important. The norm $\norm{\cdot}_{p}$ in $L^p(0,1; \R^m)$ is defined for $\xi \in L^p(0, 1; \R^m)$ by $\norm{\xi}_{p}^p = \int_0^1\abs{\xi(x)}_p^p \diff x$ if $p < \infty$ and $\norm{\xi}_{\infty} = \esssup_{x \in [0, 1]} \abs{\xi(x)}_\infty$ if $p = \infty$. Given a real interval $I$ and a normed space $J, C^0(I; J)$ denotes the set of continuous functions from $I$ to $J$, and $C^0([0, 1]; \mathbb R^m)$ is endowed with the norm $\norm{\cdot}_\infty$.

\section{PRELIMINARIES}
\label{sec:Prelim}

\subsection{Singular perturbation method}
\label{sec:SPM}

Let us briefly describe the idea of the singular perturbation method applied to \eqref{transport-ode}. As $\varepsilon \to 0^+$, we expect solutions of the transport PDE in \eqref{transport-ode} to become close to solutions of $\Lambda \hat y_x = 0$. Hence, for every $t > 0$, $x \mapsto \hat y(t, x)$ is constant, and we denote its value by $\hat y_\ast(t)$. Using the third equation of \eqref{transport-ode} and the invertibility of $I_m - G_1$, we obtain that $\hat y_\ast(t) = (I_m - G_1)^{-1} G_2 z(t)$. Inserting this into the first equation of \eqref{transport-ode}, we obtain that $\dot z(t) = (A+B(I_m-G_1)^{-1}G_2) z(t)$. This motivates the introduction of the reduced order system
\begin{equation}
\label{ros}
\left\{
\begin{aligned}
\dot{\bar z}(t) & = (A+B(I_m-G_1)^{-1}G_2) \bar z(t), & & t > 0, \\
\bar z(0) & = z_0.
\end{aligned}
\right.
\end{equation}

Let us now consider the function $\tilde y$ defined by $\tilde y(x, \tau) = y(x, \varepsilon \tau) - \hat y_\ast(\varepsilon \tau) = y(x, \varepsilon \tau) - (I_m - G_1)^{-1} G_2 z(\varepsilon \tau)$, the difference between the real solution $y$ and its expected approximation $\hat y$ in a fast time scale $\tau = t / \varepsilon$. Then $\tilde y_\tau(x, \tau) + \Lambda \tilde y_x(x, \tau) = -\varepsilon (I_m - G_1)^{-1} G_2 (A z(\varepsilon \tau) + B y(1, \varepsilon \tau))$ and the boundary condition on $\tilde y$ becomes $\tilde y(0, \tau) = G_1 \tilde y(1, \tau)$. Hence, as $\varepsilon \to 0^+$, one expects $\tilde y$ to be approximated by solutions of the boundary-layer system
\begin{equation}
\label{bls}
\left\{
\begin{aligned}
& \begin{aligned}
& \bar y_\tau(x,\tau) + \Lambda \bar y_x(x,\tau)=0, & & x \in (0, 1),\, \tau > 0, \\
& \bar y(0,\tau) = G_1 \bar y(1,\tau), & & \tau > 0,
\end{aligned} \\
& \bar y(x,0)=y_0(x)-(I_m-G_1)^{-1} G_2 z_0, \quad x \in (0, 1).
\end{aligned}
\right.
\end{equation}

\subsection{Well-posedness}

The reduced order system \eqref{ros} is a well-posed linear ODE, whose solutions belong to $C^\infty([0, \infty);\allowbreak \mathbb R^n)$.

The well-posedness of the full system \eqref{transport-ode} and the boun\-dary-layer system \eqref{bls} were studied in $L^2(0, 1; \mathbb R^m)$ in \cite[Appendix~A]{bastin2016stability}. In particular, \cite[Theorem~A.4]{bastin2016stability} shows that \eqref{bls} has a unique weak solution $\bar y \in C^0([0,\infty);L^2(0,1;\R^m))$ (in the sense of \cite[Definition~A.3]{bastin2016stability}) provided that the initial condition belongs to $L^2(0,1;\R^m)$, and \cite[Theorem~A.6]{bastin2016stability} shows that \eqref{transport-ode} has a unique solution $(z, y) \in C^0([0,\infty);\R^n \times L^2(0,1;\R^m))$ in the sense of \cite[Definition~A.5]{bastin2016stability}, provided that $z_0\in \R^n$ and $y_0\in L^2(0,1;\R^m)$.

One can also obtain the well-posedness of \eqref{bls} replacing the Hilbert space $L^2(0, 1; \mathbb R^m)$ by $L^p(0, 1; \mathbb R^m)$ for any $p \in [1, \infty]$ or by $C^0([0, 1]; \mathbb R^m)$ when exploring its link with continuous-time difference equations. This was done in \cite{Chitour2016Stability} in $L^p(0, 1; \mathbb R^m)$, which proves well-posedness in its Proposition~4.6, in the sense of its Definition~4.1, while the well-posedness in $C^0([0, 1]; \mathbb R^m)$, under the compatibility condition $\bar y(0, 0) = G_1 \bar y(1, 0)$ on the initial condition, can be proved by combining the equivalence between \eqref{bls} and a difference equation established in \cite[Proposition~4.2]{Chitour2016Stability} and the well-posedness result for the latter from \cite[Chapter~9]{Hale1993Introduction}.

As for the well-posedness of \eqref{transport-ode} in spaces other than the Hilbert space $\mathbb R^n \times L^2(0, 1; \mathbb R^m)$, one can make use of the techniques transforming transport equations into delay systems, such as those used in \cite[Chapters~2 and 3]{bastin2016stability}. More precisely, denoting by $y_i$ the $i$th component of $y$ and by $\lambda_i$ the $i$th diagonal entry of $\Lambda$, the method of characteristics shows that, for every $i \in \{1, \dotsc, m\}$, the $i$th component of the transport equation of \eqref{transport-ode} is equivalent to having
\[y_i(x + \tfrac{\lambda_i}{\varepsilon} h, t + h) = y_i(x, t) \; \text{for all }t \geq 0, h \geq 0, x \in [0, 1].\]
In particular, we have $y_i(1, t) = y_i(0, t-\tfrac{\varepsilon}{\lambda_i})$, and hence \eqref{transport-ode} can be rewritten in terms of the functions $z(\cdot)$ and $y(0, \cdot)$ as the delay system of neutral type
\begin{equation}
\label{eq:delay-neutral}
\left\{
\begin{aligned}
    \dot z(t) & \textstyle = A z(t) + \sum_{i=1}^m B e_i e_i^T y(0, t - \tfrac{\varepsilon}{\lambda_i}), \\
    y(0, t) & \textstyle = \sum_{i=1}^m G_1 e_i e_i^T y(0, t - \tfrac{\varepsilon}{\lambda_i}) + G_2 z(t),
\end{aligned}
\right.
\end{equation}
where $e_1, \dotsc, e_m$ are the vectors of the canonical basis of $\mathbb R^m$. The well-posedness of \eqref{transport-ode} in $\mathbb R^n \times C^0([0, 1]; \mathbb R^m)$ can hence be deduced from the classical theory of neutral functional differential equations, and can be found in \cite[Chapter~9]{Hale1993Introduction}, under the compatibility condition $y_0(0) = G_1 y_0(1) + G_2 z_0$, while the well-posedness in $\mathbb R^n \times L^p(0, 1; \mathbb R^m)$ is covered by the results of \cite{Burns1983Linear}.

\subsection{Exponential stability}

The reduced order system \eqref{ros} is said to be exponentially stable in norm $p$ if there exist positive constants $C$ and $\nu$ such that $\abs{\bar z(t)}_p \leq C e^{-\nu t} \abs{\bar z(0)}_p$ for every solution $\bar z$ of \eqref{ros} and $t \in [0, \infty)$. Since all norms are equivalent in finite dimension, the notion of exponential stability turns out to be independent of $p$. Moreover, \eqref{ros} is exponentially stable if and only if all complex roots of its characteristic functions $\Delta_{\ros}$ have negative real part, where
\begin{equation}
\label{char-ros}
\Delta_{\ros}(s)= \det (sI_n - A-B(I_m-G_1)^{-1} G_2).
\end{equation}

Let us now consider the notions of exponential stability for \eqref{transport-ode} and \eqref{bls}.

\begin{definition}
\label{def:expo-stability}
Let $\mathcal B$ denote one of the Banach spaces $C^0([0, 1]; \mathbb R^m)$ or $L^p(0, 1; \mathbb R^m)$ for some $p \in [1, \infty]$. We say that \eqref{transport-ode} is exponentially stable in $\R^n\times \mathcal{B}$ if there exists positive constants $C,\nu$ such that, for every $z_0\in \mathbb{R}^n$ and $y_0\in \mathcal{B}$, the corresponding solution $(z,y)\colon [0,\infty) \to \mathbb{R}^n \times \mathcal{B}$ of \eqref{transport-ode} satisfies, for every $t \geq 0$,
\begin{equation}
\label{b-stab}
\abs{z(t)} + \norm{y(\cdot,t)}_\mathcal{B} \leq C e^{-\nu t}(\abs{z_0} + \norm{y_0}_\mathcal{B}),
\end{equation}
where $\abs{\cdot}$ is an arbitrary norm in $\mathbb R^n$. Analogously, we say that \eqref{bls} is exponentially stable in $\mathcal{B}$ if there exists positive constants $C,\nu$ such that, for every initial condition in $\mathcal{B}$, the corresponding solution $\bar y\colon [0,\infty) \to \mathcal{B}$ of \eqref{bls} satisfies \eqref{b-stab} with $z=0$ and $z_0 = 0$.
\end{definition}

Even though the definition of exponential stability of \eqref{bls} depends on the space $\mathcal B$, it turns out that all these definitions are equivalent for the spaces $\mathcal B$ from Definition~\ref{def:expo-stability}, as discussed in \cite[Remark~3.6]{bastin2016stability} (see also \cite[Corollary~4.11]{Chitour2016Stability} for the equivalence in $L^p$ spaces, $p \in [1, \infty]$). For this reason, in the sequel, we will say that \eqref{bls} is exponentially stable if it is exponentially stable in some (and hence every) space $\mathcal B$ among $C^0([0, 1]; \mathbb R^m)$ or $L^p(0, 1; \mathbb R^m)$, $p \in [1, \infty]$. Exponential stability of \eqref{bls} can also be characterized through the characteristic function of \eqref{bls}, $\Delta_{\bls}$, given by
\begin{equation}
\label{char-bls}
\Delta_{\bls}(s)=\det (I_m - e^{-s\Lambda^{-1}}G_1),
\end{equation}
as stated in the next result, taken from \cite[Theorem~3.5 and Remark~3.6]{bastin2016stability}.

\begin{prop}
\label{prop:stab_bls}
System \eqref{bls} is exponentially stable if and only if there exists $\alpha > 0$ such that all roots $s$ of $\Delta_{\bls}$ satisfy $\Real s \leq -\alpha$.
\end{prop}

Similarly to \eqref{bls}, exponential stability of \eqref{transport-ode} is also independent of the space $\mathcal B$ chosen among the ones in Definition~\ref{def:expo-stability}. Indeed, \eqref{transport-ode} is equivalent to \eqref{eq:delay-neutral}, and the independence of the exponential stability of the latter on the space $\mathcal B$ was established in \cite[Corollary~II.1.8]{Kaashoek1992Characteristic} for the $L^p$ spaces, obtaining the same criterion of exponential stability as \cite[Chapter~9]{Hale1993Introduction} for $C^0$ (see also \cite{Hale2002Strong}). Hence, as for \eqref{bls}, we will say that \eqref{transport-ode} is exponentially stable if it is exponentially stable in some (and hence every) space $\mathbb R^n \times \mathcal B$ with $\mathcal B$ among $C^0([0, 1]; \mathbb R^m)$ or $L^p(0, 1; \mathbb R^m)$, $p \in [1, \infty]$. Taking the Laplace transform of \eqref{eq:delay-neutral}, the characteristic function of \eqref{transport-ode} is
\begin{equation}
\label{characteristic}
\Delta_{\full}(s, \varepsilon) = \det M(s, \varepsilon),
\end{equation}
where
\begin{equation}
\label{transfer}
M(s,\varepsilon) = \begin{bmatrix}
s I_n - A & B e^{-\Lambda^{-1} \varepsilon s}\\
 G_2 & I_m - G_1 e^{-\Lambda^{-1} \varepsilon s}
\end{bmatrix}.
\end{equation}
From \cite[Theorem~3.14]{bastin2016stability}, we have the following result relating stability of \eqref{transport-ode} and roots of $\Delta_{\full}$.

\begin{prop}
\label{prop:stab_full}
System \eqref{transport-ode} is exponentially stable if and only if there exists $\alpha > 0$ such that all roots $s$ of $\Delta_{\full}$ satisfy $\Real s \leq -\alpha$.
\end{prop}

The roots of $\Delta_{\full}$, $\Delta_{\ros}$, and $\Delta_{\bls}$ are referred to in the sequel as \emph{poles} of systems \eqref{transport-ode}, \eqref{ros}, and \eqref{bls}, respectively.

\section{MAIN RESULT}
\label{sec:MainRes}

We are now able to state the main result of this article.

\begin{teo}
\label{full_stab}
Suppose that \eqref{ros} and \eqref{bls} are exponentially stable. Then there exists $\varepsilon_\ast > 0$ such that \eqref{transport-ode} is exponentially stable for every $\varepsilon \in (0, \varepsilon_\ast]$.
\end{teo}

The proof of Theorem~\ref{full_stab} relies on the spectral characterizations of exponential stability from Propositions~\ref{prop:stab_bls} and \ref{prop:stab_full}. The first step is to provide a suitable rewriting of the characteristic function $\Delta_\full$. For that purpose, we define, for each $s\in \mathbb{C}$ and $\varepsilon >0$ such that $\Delta_\bls(\varepsilon s) \neq 0$,
\begin{equation*}
M_1(s,\varepsilon) = sI_n-A- B e^{-\Lambda^{-1}\varepsilon s}(I_m-G_1 e^{-\Lambda^{-1} \varepsilon s})^{-1}G_2.
\end{equation*}
We now remark that, for such $s$ and $\varepsilon$, we have
\begin{multline}
\label{schur}
M(s,\varepsilon) \begin{bmatrix}
I_n & 0_{n\times m} \\ -(I_m-G_1e^{-\Lambda^{-1}\varepsilon s})^{-1}G_2 & I_m
\end{bmatrix}, \vspace{.1cm}\\ = \begin{bmatrix}
M_1(s,\varepsilon)& B e^{-\Lambda^{-1} \varepsilon s} \\ 0_{m \times n} & I_m - G_1 e^{-\Lambda^{-1} \varepsilon s}
\end{bmatrix}.
\end{multline}
Taking the determinant of the above identity, we deduce that
\begin{equation}
\label{roots}
\Delta_\full (s,\varepsilon) = \det M_1(s,\varepsilon) \Delta_{\bls}(\varepsilon s).
\end{equation}

To explain the idea of the proof of Theorem~\ref{full_stab}, let us note that, formally, taking $\varepsilon=0$ in \eqref{roots}, we obtain that $\Delta_\full(s, 0) = \det M_1(s, 0) \Delta_\bls(0)$. Since $\det M_1(s,0) = \Delta_{\ros}(s)$ and $\Delta_{\bls}(0) = \det(I_m - G_1) \neq 0$, we have that the roots of $\Delta_{\full}(\cdot, 0)$ coincide with those of $\Delta_{\ros}(\cdot)$. If the behavior of the system is continuous with respect to $\varepsilon$ as $\varepsilon \to 0^+$, we may then expect that the asymptotic behavior of \eqref{transport-ode} is similar to  the asymptotic behaviour of \eqref{ros} as $\varepsilon \to 0^+$. The main idea of the proof is to show that this expected behavior is indeed correct. More precisely, we will prove the following result.

\begin{prop}
\label{prop:main-spectral}
Suppose that \eqref{ros} and \eqref{bls} are exponentially stable and let $\alpha > 0$ be such that $\Real s < -\alpha$ for every root $s$ of $\Delta_\ros$. Then there exists $\varepsilon_\ast > 0$ such that, for every $\varepsilon \in (0, \varepsilon_\ast]$, we have $\Real s < -\alpha$ for every root $s$ of $s \mapsto \Delta_\full(s, \varepsilon)$.
\end{prop}

Due to Proposition~\ref{prop:stab_full}, Theorem~\ref{full_stab} will follow as an immediate corollary of Proposition~\ref{prop:main-spectral} once the latter is proved. Note that the existence of an $\alpha$ as in the statement of Proposition~\ref{prop:main-spectral} is guaranteed by the exponential stability of \eqref{ros}.

The proof of Proposition~\ref{prop:main-spectral} is based in the decomposition \eqref{roots} and it is split into several technical lemmas. We start by the following property of the roots of $\Delta_\bls$.

\begin{lem}
\label{lem:lower-bound-bls}
Assume that \eqref{bls} is exponentially stable. Then there exist $\tilde\alpha > 0$ and $\kappa > 0$ such that $\Real s \leq -2 \tilde\alpha$ for every root $s$ of $\Delta_\bls$ and $\abs{\Delta_\bls(s)} \geq \kappa$ for every $s \in \mathbb C$ with $\Real s \geq -\tilde\alpha$.
\end{lem}

\begin{proof}
The existence of $\tilde\alpha > 0$ such that $\Real s \leq -2 \tilde\alpha$ for every root $s$ of $\Delta_\bls$ is a straightforward consequence of Proposition~\ref{prop:stab_bls}. In addition, we have that
\[
\lim_{\Real s \to +\infty} \Delta_\bls(s) = \det I_m = 1,
\]
and the above limit is uniform in $\Imag s$. Hence, there exists $\zeta > 0$ such that $\abs{\Delta_\bls(s)} \geq \tfrac{1}{2}$ for every $s \in \mathbb C$ with $\Real s \geq \zeta$.

Note that $\Delta_\bls$ is an exponential polynomial (in the sense of \cite{avellar1980zeros}), i.e., it can be written under the form $\Delta_\bls(s) = 1 + \sum_{j=1}^N a_j e^{-s \tau \cdot \gamma_j}$, where $N$ is a positive integer, $a_1, \dotsc, a_N$ are nonzero real coefficients, $\tau = (1/\lambda_1, \dotsc, 1/\lambda_m)$, $\lambda_1,\allowbreak \dotsc, \lambda_m$ are the diagonal entries of $\Lambda$, and $\gamma_1, \dotsc, \gamma_N$ are vectors in $\mathbb N^m$, where $\mathbb N$ denotes the set of nonnegative integers. In addition, for every $s \in \mathbb C$ with $-\tilde\alpha \leq \Real s \leq \zeta$, $s$ is at a distance of at least $\tilde\alpha > 0$ from every zero of $\Delta_\bls$, in the usual metric of $\mathbb C$. Hence, using \cite[Lemma~2.1(ii)]{avellar1980zeros}, there exists $\tilde\kappa > 0$ such that $\abs{\Delta_\bls(s)} \geq \tilde\kappa$ for every $s \in \mathbb C$ with $-\tilde\alpha \leq \Real s \leq \zeta$. The conclusion of the lemma follows by taking $\kappa = \min\{\tfrac{1}{2}, \tilde\kappa\}$.
\end{proof}

Using Lemma~\ref{lem:lower-bound-bls}, we obtain the following property on the behavior of $s \mapsto \Delta_\bls(\varepsilon s)$ for $\varepsilon$ small.

\begin{lem}
\label{lem:lower-bound-bls-epsilon}
Assume that \eqref{ros} and \eqref{bls} are exponentially stable and let $\alpha > 0$ be as in Proposition~\ref{prop:main-spectral}. Then there exist $\bar\varepsilon > 0$ and $\kappa > 0$ such that, for every $\varepsilon \in (0, \bar\varepsilon]$, we have $\Real s \leq - 2 \alpha$ for every root $s$ of $s \mapsto \Delta_\bls(\varepsilon s)$, and $\abs{\Delta_\bls(\varepsilon s)} \geq \kappa$ for every $s \in \mathbb C$ with $\Real s \geq -\alpha$.
\end{lem}

\begin{proof}
Let $\tilde\alpha > 0$ and $\kappa > 0$ be given by Lemma~\ref{lem:lower-bound-bls}, set $\bar\varepsilon = \frac{\tilde\alpha}{\alpha}$, and let $\varepsilon \in (0, \bar\varepsilon]$. If $s$ is a root of $s \mapsto \Delta_\bls(\varepsilon s)$, then $\varepsilon s$ is a root of $\Delta_\bls$, thus $\Real (\varepsilon s) \leq - 2 \tilde\alpha$, implying that $\Real s \leq -\frac{2 \tilde\alpha}{\varepsilon} \leq -2\alpha$, as required. In addition, if $s \in \mathbb C$ is such that $\Real s \geq -\alpha$, then $\Real (\varepsilon s) \geq - \varepsilon \alpha \geq -\tilde\alpha$, and thus $\abs{\Delta_\bls(\varepsilon s)} \geq \kappa$ by Lemma~\ref{lem:lower-bound-bls}.
\end{proof}

We now turn to the analysis of the first term of \eqref{roots}, i.e., of $\det M(s, \varepsilon)$.

\begin{lem}
\label{bounded_pole}
Assume that \eqref{ros} and \eqref{bls} are exponentially stable, let $\alpha > 0$ be as in Proposition~\ref{prop:main-spectral}, and $\bar\varepsilon > 0$ be given by Lemma~\ref{lem:lower-bound-bls-epsilon}. Then there exists $R > 0$ such that, for every $\varepsilon \in (0, \bar\varepsilon]$, we have $\abs{s} \leq R$ for every root $s$ of $s \mapsto \det M(s, \varepsilon)$ with $\Real s \geq -\alpha$.
\end{lem}

\begin{proof}
Thanks to Lem\-ma~\ref{lem:lower-bound-bls-epsilon}, $I_m - G_1 e^{-\Lambda^{-1} \varepsilon s}$ is invertible for every $\varepsilon \in (0, \bar\varepsilon]$ and $s \in \mathbb C$ with $\Real s \geq -\alpha$, and thus $M_1(s, \varepsilon)$ is well-defined for all such $\varepsilon$ and $s$.

Fix $\varepsilon \in (0, \bar\varepsilon]$ and let $s \in \mathbb C$ be such that $\det M_1(s, \varepsilon) = 0$ and $\Real s \geq -\alpha$. Since $\det M_1(s, \varepsilon) = 0$, we have that $s$ is an eigenvalue of the matrix $A+ Be^{-\Lambda^{-1}\varepsilon s} (I_m-G_1e^{-\Lambda^{-1}\varepsilon s})^{-1} G_2$, and thus
\begin{equation}
\label{eq:estim-s}
\begin{aligned}
\abs{s} & \leq \rho(A + B e^{-\Lambda^{-1}\varepsilon s} (I_m - G_1 e^{-\Lambda^{-1} \varepsilon s})^{-1} G_2), \\
& \leq \abs*{A}_1+ \abs*{B e^{-\Lambda^{-1}\varepsilon s}}_1 \abs*{(I_m - G_1 e^{-\Lambda^{-1}\varepsilon s})^{-1}}_1 \abs*{G_2}_1,
\end{aligned}
\end{equation}
which comes from the fact that the $\rho(M)\leq \abs{M}$ for any induced matrix norm $\abs{\cdot}$. Note that
\begin{equation}
\label{eq:estim-B}
\abs*{B e^{-\Lambda^{-1}\varepsilon s}}_1 = \max_{j\in \{1, \dotsc,m\}} \abs{B_j}_1 \abs*{e^{-\frac{\varepsilon s}{\lambda_j}}} \leq \abs{B}_1 e^{\frac{\bar\varepsilon \alpha}{\lambda_{\min}}},
\end{equation}
where $B_1, \dotsc, B_m$ denote the columns of $B$ and $\lambda_{\min}$ is the smallest diagonal entry of $\Lambda$. On the other hand, using \cite[Chapter~1, (4.12)]{kato1980short}, there exists a constant $C > 0$ depending only on the dimension $m$ such that
\begin{equation*}
\abs{(I_m - G_1 e^{-\Lambda^{-1}\varepsilon s})^{-1}}_1 \leq C \frac{\abs{I_m - G_1 e^{-\Lambda^{-1}\varepsilon s}}_1^{m-1}}{\det (I_m - G_1 e^{-\Lambda^{-1}\varepsilon s})}.
 \end{equation*}
Similarly to the estimate of $\abs{B e^{-\Lambda^{-1}\varepsilon s}}_1$, we obtain that
\begin{equation}
\label{eq:estim-G1}
\abs{(I_m - G_1 e^{-\Lambda^{-1} \varepsilon s})^{-1}}_1 \leq \tfrac{C}{\kappa} \left(1 + \abs{G_1}_1 e^{\frac{\bar\varepsilon \alpha}{\lambda_{\min}}}\right)^{m-1},
\end{equation}
where $\kappa$ is the constant from Lemma~\ref{lem:lower-bound-bls-epsilon}. Combining \eqref{eq:estim-s}, \eqref{eq:estim-B}, and \eqref{eq:estim-G1}, we obtain the conclusion.
\end{proof}

We are now in position to prove Proposition~\ref{prop:main-spectral} and conclude the proof of Theorem~\ref{full_stab}.

{\renewcommand{\proof}{\noindent\hspace{2em}{\itshape Proof of Proposition~\ref{prop:main-spectral}: }}
\begin{proof}
Let $\bar\varepsilon$ and $R$ be given by Lemma~\ref{bounded_pole} and assume, to obtain a contradiction, that there exist a sequence of positive real numbers $(\varepsilon_k)_{k \in \mathbb N}$ and a sequence of complex numbers $(s_k)_{k \in \mathbb N}$ such that $\varepsilon_k \to 0$ as $k \to \infty$ and, for every $k \in \mathbb N$, we have $\Delta_\full(s_k, \varepsilon_k) = 0$ and $\Real s_k \geq -\alpha$.

Up to excluding finitely many terms of the sequence, we have $\varepsilon_k \in (0, \bar\varepsilon]$ for every $k \in \mathbb N$, and thus, by Lemma~\ref{bounded_pole}, we have $\abs{s_k} \leq R$ for every $k \in \mathbb N$. Hence, up to extracting a subsequence (which we still denote using the same notation for simplicity), there exists $s_\ast$ such that $s_k \to s_\ast$ as $k \to \infty$. In particular, $\Real s_\ast \geq -\alpha$.

Note that, by Lemma~\ref{lem:lower-bound-bls-epsilon}, we have $\Delta_\bls(\varepsilon_k s_k) \neq 0$ for every $k \in \mathbb N$ and thus, using \eqref{roots} we obtain that $\det M_1(s_k, \varepsilon_k) = 0$ for every $k \in \mathbb N$. Since $M_1$ is continuous on the set $\{s \in \mathbb C \text{ : } \Real s \geq -\alpha\} \times [0, \bar\varepsilon]$, we deduce, letting $k \to +\infty$, that $\det M_1(s_\ast, 0) = 0$. Since $\Delta_\ros(\cdot) = \det M_1(\cdot, 0)$, we deduce that $\Delta_\ros(s_\ast) = 0$, which is a contradiction since $\Real s_\ast \geq -\alpha$ but $\Real s < -\alpha$ for every root $s$ of $\Delta_\ros$, by definition of $\alpha$. This contradiction establishes the result.
\end{proof}
}

\section{COMPARISON WITH RESPECT TO THE LYAPUNOV APPROACH}
\label{sec:Comparison}

We now compare our main result, Theorem~\ref{full_stab}, to \cite[Theorem~1]{tang2017stability}, which is the main result of that reference for the singular perturbation of \eqref{transport-ode}. In this section, we shall denote by $\mathcal D^+_m$ be the set of $m \times m$ diagonal matrices with positive diagonal entries.

\begin{teo}[{\cite[Theorem~1]{tang2017stability}}]
\label{thm:tang2017}
Suppose that \eqref{ros} is exponentially stable and there exist $\mu > 0$ and $Q \in \mathcal D^+_m$ such that $e^{-\mu} Q \Lambda - G_1^T Q \Lambda G_1$ is positive definite. Then there exists $\varepsilon_\ast > 0$ such that \eqref{transport-ode} is exponentially stable in $\mathbb R^n \times L^2(0, 1; \mathbb R^m)$ for every $\varepsilon \in (0, \varepsilon_\ast]$.
\end{teo}

The existence of $\mu > 0$ and $Q \in \mathcal D^+_m$ such that $e^{-\mu} Q \Lambda - G_1^T Q \Lambda G_1$ is positive definite is shown in \cite[Proposition~2]{tang2017stability} to imply the exponential stability of \eqref{bls}. Hence, Theorem~\ref{thm:tang2017} can be seen as a particular case of Theorem~\ref{full_stab}: whenever the assumptions of Theorem~\ref{thm:tang2017} are satisfied, the assumptions of Theorem~\ref{full_stab} are satisfied as well.

Before proving that Theorem~\ref{full_stab} is strictly stronger than Theorem~\ref{thm:tang2017}, let us provide the following criterion.

\begin{lem}
\label{lem:rho2}
Let $G_1$ be an $m \times m$ matrix with real entries and $\Lambda \in \mathcal D^+_m$. The following assertions are equivalent.
\begin{enumerate}
\item\label{lem:rho2_1} There exist $\mu > 0$ and $Q \in \mathcal D^+_m$ such that $e^{-\mu} Q \Lambda - G_1^T Q \Lambda G_1$ is positive definite.
\item\label{lem:rho2_2} $\inf_{D \in \mathcal D^+_m} \norm{D G_1 D^{-1}}_2 < 1$.
\end{enumerate}
\end{lem}

\begin{proof}
Note that \ref{lem:rho2_2}) holds true if and only if there exists $D \in \mathcal D^+_m$ such that $\norm{D G_1 D^{-1}}_2 < 1$ and, using the expression of the induced matrix $2$-norm, the latter condition is equivalent to $\rho(M) < 1$, where $M=D^{-1}G^T_1 D^2 G_1D^{-1}$.
Since $M=(D^{-1} G_1^T D)(D^{-1} G_1^T D)^T$, it follows that $M$ is a symmetric positive semi-definite matrix, and thus there exists a real orthogonal matrix $U$ such that $U M U^T$ is diagonal, and its diagonal entries are the eigenvalues of $M$. Hence, $\rho(M) < 1$ if and only if there exists $\mu > 0$ such that $e^{-\mu} I_m - U M U^T$ is positive definite. Multiplying by $U^T$ on the left and by $U$ on the right, we deduce that the latter condition is equivalent to $e^{-\mu} I_m - M$ being positive definite, which, up to multiplying by $D^T = D$ on the left and $D$ on the right, is equivalent to $e^{-\mu} D^2 - G_1^T D^2 G_1$ being positive definite. We have thus shown that \ref{lem:rho2_2}) holds true if and only if there exists $D \in \mathcal D^+_m$ and $\mu > 0$ such that $e^{-\mu} D^2 - G_1^T D^2 G_1$ is positive definite. The conclusion follows by imposing the relation $Q \Lambda = D^2$ between $Q$, $\Lambda$, and $D$, which is possible since these three matrices belong to $\mathcal D^+_m$.
\end{proof}

\begin{rmk}
The quantity $\inf_{D \in \mathcal D^+_m} \norm{D G_1 D^{-1}}_2$ from assertion \ref{lem:rho2_2}) of Lemma~\ref{lem:rho2} is usually denoted in the literature by $\rho_2(G_1)$, and the condition $\rho_2(G_1) < 1$ is a sufficient condition for the stability in $L^2$ of \eqref{bls} (see, e.g., \cite[Section~3.1]{bastin2016stability}).
\end{rmk}

We now provide an example showing that Theorem~\ref{full_stab} is strictly stronger than Theorem~\ref{thm:tang2017}, i.e., that there exist $\Lambda \in \mathcal D^+_m$ and a matrix $G_1$ of size $m \times m$ such that \eqref{bls} is exponentially stable but, for every $\mu > 0$ and $Q \in \mathcal D^+_m$, the matrix $e^{-\mu} Q \Lambda - G_1^T Q \Lambda G_1$ is not positive definite.

\begin{ex}
\label{expl:better}
Let $m = 2$ and consider
\[
G_1 = \begin{bmatrix}
1 & -2 \\
\frac{1}{4} & -\frac{1}{2} \\
\end{bmatrix}, \qquad \Lambda = \begin{bmatrix}
1 & 0 \\
0 & \frac{1}{2} \\
\end{bmatrix}.
\]
In this case, the characteristic function \eqref{char-bls} of the boundary-layer system \eqref{bls} is
\[
\Delta_{\bls}(s) = 1 - e^{-s} + \tfrac{1}{2} e^{-2s},
\]
whose roots $s$ satisfy $e^s = \tfrac{\sqrt{2}}{2} e^{\pm i \frac{\pi}{4}}$, i.e., the set of roots of $\Delta_{\bls}$ is $\left\{-\tfrac{1}{2} \ln 2 \pm i (\tfrac{\pi}{4} + 2 k \pi) \text{ : } k \in \mathbb Z\right\}$. Since all these roots have real part equal to $-\tfrac{1}{2} \ln 2$, the boundary-layer system \eqref{bls} is exponentially stable, and hence the assumption on \eqref{bls} from Theorem~\ref{full_stab} is satisfied.

On the other hand, for every $D \in \mathcal D^+_2$, denoting by $d_1$, $d_2$ the diagonal entries of $D$ and $\delta = \frac{d_1}{d_2}$, we have
\[
D^{-1} G_1^T D^2 G_1 D^{-1} = \begin{bmatrix}
1 + \tfrac{1}{16 \delta^2} & - 2 \delta - \tfrac{1}{8 \delta} \\
-2 \delta - \tfrac{1}{8\delta} & 4 \delta^2 + \tfrac{1}{4}.
\end{bmatrix}
\]
The eigenvalues of the above matrix are $0$ and $\tfrac{5}{4} + 4 \delta^2 + \tfrac{1}{16 \delta^2}$, and thus $\norm{D G_1 D^{-1}}_2^2 = \rho(D^{-1} G_1^T D^2 G_1 D^{-1}) = \tfrac{5}{4} + 4 \delta^2 + \tfrac{1}{16 \delta^2} > 1$ for every $D \in \mathcal D^+_2$. Hence, using Lemma~\ref{lem:rho2}, we conclude that the assumption on $G_1$ and $\Lambda$ from Theorem~\ref{thm:tang2017} is not satisfied.
\end{ex}

\section{NUMERICAL ILLUSTRATION}
\label{sec:Numerical}

As an illustration of Theorem~\ref{full_stab}, we provide a numerical si\-mu\-lation\footnote{Simulations were done in Python, with an explicit computation of the solution of \eqref{ros}, an upwind scheme in space and an explicit Euler scheme in time for the hyperbolic PDEs of \eqref{transport-ode} and \eqref{bls}, and an explicit Euler scheme in time for the ODE in \eqref{transport-ode}. The time step was adapted according to the value of $\varepsilon$ and the space step was chosen in order to satisfy a CFL stability condition for the numerical scheme. The simulation code is available in {\scriptsize\url{https://gitlab.inria.fr/mazanti/singular-perturbation-fast-transport-pde-and-slow-ode}}.} of \eqref{transport-ode} when the parameters $G_1$ and $\Lambda$ of the boun\-dary-layer system are those of Example~\ref{expl:better} and with
\begin{gather*}
n = 1,\quad A = 2,\quad B = \begin{bmatrix}1 & 2 \\\end{bmatrix},\quad G_2 = \begin{bmatrix}-1 & 0\end{bmatrix}^T, \\
z_0 = 1,\quad y_0(x) = \begin{bmatrix}-\cos\left(\tfrac{5\pi}{2} x\right) & 0 \end{bmatrix}^T,
\end{gather*}
for different values of $\varepsilon$. Note that \eqref{bls} is exponentially stable with this choice of parameters, as shown in Example~\ref{expl:better}, and \eqref{ros} reads $\dot{\bar z}(t) = -2 \bar z(t)$, so it is also exponentially stable. Hence, the assumptions of Theorem~\ref{full_stab} are satisfied. In addition, the initial conditions were chosen so that the compatibility condition $y_0(0) = G_1 y_0(1) + G_2 z_0$ required for the existence of solutions in $C^0$ is satisfied.

\begin{figure}[ht]
\centering
\resizebox{0.75\columnwidth}{!}{\input{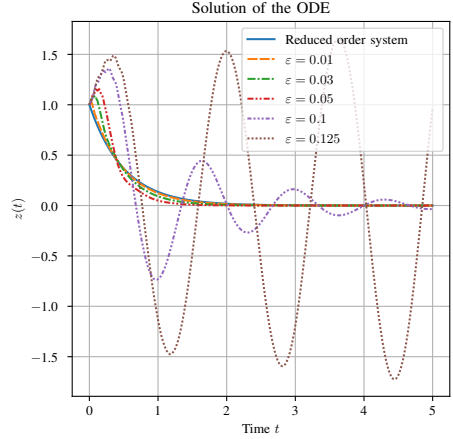}}
\caption{Solution of the ODE in \eqref{transport-ode} for various values of $\varepsilon$ and solution of the reduced order system \eqref{ros}.}
\label{fig:ode}
\end{figure}

The solution of the ODE in \eqref{transport-ode} is represented in Fig.~\ref{fig:ode} for various values of $\varepsilon$, together with the solution of the reduced order system \eqref{ros}. We observe that, for $\varepsilon$ large, the system \eqref{transport-ode} seems unstable, as $z$ oscillates with an increasing amplitude, but stability is recovered as $\varepsilon$ is reduced, and the behavior of $z$ becomes close to that of the solution $\bar z$ of \eqref{ros}.

\begin{figure*}
\begin{tabular}{@{} >{\centering} m{0.4\textwidth} @{} >{\centering} m{0.6\textwidth} @{}}
\resizebox{0.36\textwidth}{!}{\input{sol_pde_trace.pgf}} \vspace*{-1.5em} &
\begin{tikzpicture}
\node[inner sep=0pt] at (0, 0) {\includegraphics[width=0.58\textwidth]{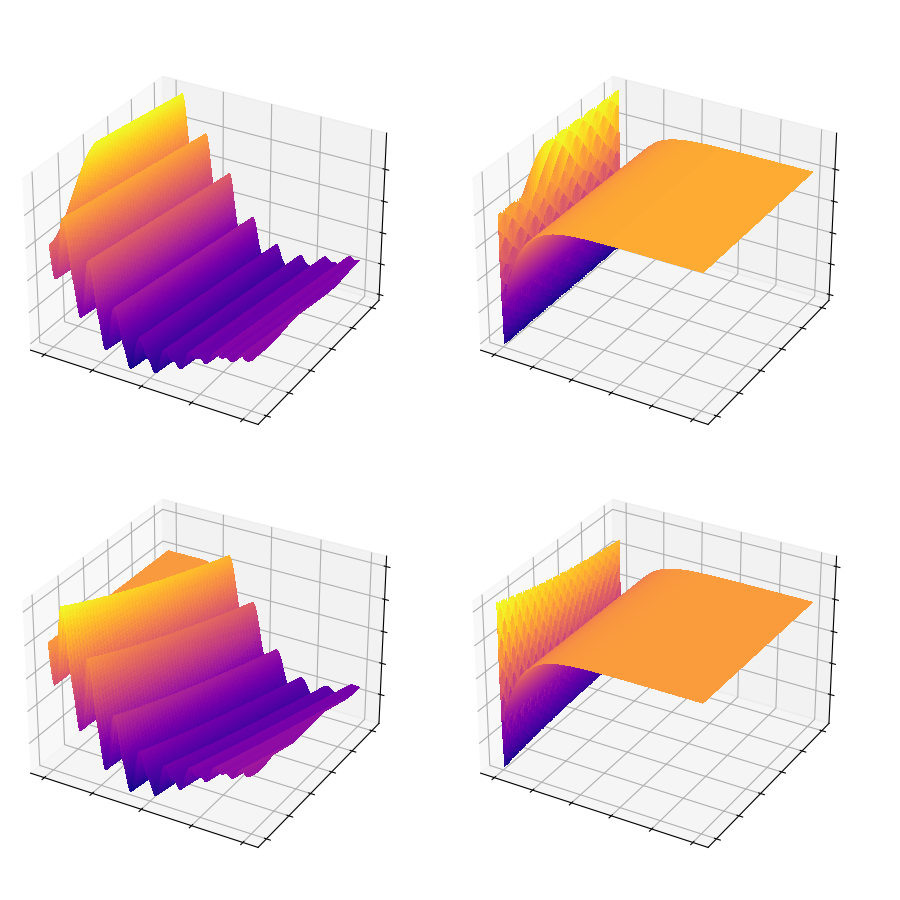}};
\node[inner sep=0pt] at (0, 0) {\resizebox{0.58\textwidth}{!}{\input{sol_pde.pgf}}};
\end{tikzpicture} \tabularnewline
(a) & (b)
\end{tabular}
\caption{(a) Trace of the solution of the PDE in \eqref{transport-ode} at $x = 0$ for various values of $\varepsilon$, compared with the function $\bar y_\ast(t) = (I_m - G_1)^{-1} G_2 \bar z(t)$. (b) Solution of the PDE in \eqref{transport-ode} with $\varepsilon = 0.01$, in the time intervals $[0, 8\varepsilon]$ (left) and $[0, 5]$ (right). In both (a) and (b), the top and bottom figures represent the components $y_1$ and $y_2$ of $y$, respectively.}
\label{fig:pde}
\end{figure*}

Fig.~\ref{fig:pde}(a) represents the trace at $x = 0$ of the solution of the PDE in \eqref{transport-ode} for the same values of $\varepsilon$ as those from Fig.~\ref{fig:ode}. Following the discussion in Section~\ref{sec:SPM}, we expect $y(t, x)$ to be close, when $\varepsilon$ is small enough, to $\bar y_\ast(t) = (I_m - G_1)^{-1} G_2 \bar z(t)$ (which is constant in $x$), and so the latter function is also represented in Fig.~\ref{fig:pde}(a). As in Fig.~\ref{fig:ode}, we observe an unstable behavior for $\varepsilon$ large, but it becomes stable as $\varepsilon$ is small. For small $\varepsilon$, we observe strong oscillations of the solution for small times, corresponding to the fast dynamics of the system. Fig.~\ref{fig:pde}(b) represents the solution of the PDE in \eqref{transport-ode} when $\varepsilon = 0.01$, detailing, in the left, the behavior of such solutions in the small time interval $[0, 8\varepsilon]$. We observe that, at $t = 8 \varepsilon$, the solution $y$ is close to being constant in space, as expected following the discussion in Section~\ref{sec:SPM}.

\section{CONCLUSION}
\label{sec:Concl}

In this article we have applied the singular perturbation method to system \eqref{transport-ode}, showing that it is exponentially stable as soon as both subsystems \eqref{ros} and \eqref{bls} are exponentially stable and $\varepsilon$ is small enough. This result was proved using a suitable analysis of the behavior of the spectrum of \eqref{transport-ode} when $\varepsilon$ is small enough, through the spectral information of \eqref{ros} and \eqref{bls}. Moreover, our approach improves the previous result from \cite[Theorem~1]{tang2017stability}, which had been obtained through a Lyapunov analysis: the condition on \eqref{bls} of our main result, Theorem~\ref{full_stab}, is less restrictive than that from \cite[Theorem~1]{tang2017stability}.

\bibliographystyle{IEEEtran}
\bibliography{references}

\end{document}

%% file: sol_pde.pgf
%% Creator: Matplotlib, PGF backend
%%
%% To include the figure in your LaTeX document, write
%%   \input{<filename>.pgf}
%%
%% Make sure the required packages are loaded in your preamble
%%   \usepackage{pgf}
%%
%% Also ensure that all the required font packages are loaded; for instance,
%% the lmodern package is sometimes necessary when using math font.
%%   \usepackage{lmodern}
%%
%% Figures using additional raster images can only be included by \input if
%% they are in the same directory as the main LaTeX file. For loading figures
%% from other directories you can use the `import` package
%%   \usepackage{import}
%%
%% and then include the figures with
%%   \import{<path to file>}{<filename>.pgf}
%%
%% Matplotlib used the following preamble
%%   
%%   \usepackage{fontspec}
%%   \setmainfont{DejaVuSerif.ttf}[Path=\detokenize{C:/ProgramData/Anaconda3/lib/site-packages/matplotlib/mpl-data/fonts/ttf/}]
%%   \setsansfont{DejaVuSans.ttf}[Path=\detokenize{C:/ProgramData/Anaconda3/lib/site-packages/matplotlib/mpl-data/fonts/ttf/}]
%%   \setmonofont{DejaVuSansMono.ttf}[Path=\detokenize{C:/ProgramData/Anaconda3/lib/site-packages/matplotlib/mpl-data/fonts/ttf/}]
%%   \makeatletter\@ifpackageloaded{underscore}{}{\usepackage[strings]{underscore}}\makeatother
%%
\begingroup%
\makeatletter%
\begin{pgfpicture}%
\pgfpathrectangle{\pgfpointorigin}{\pgfqpoint{9.000000in}{9.000000in}}%
\pgfusepath{use as bounding box, clip}%
\begin{pgfscope}%
\definecolor{textcolor}{rgb}{0.000000,0.000000,0.000000}%
\pgfsetstrokecolor{textcolor}%
\pgfsetfillcolor{textcolor}%
\pgftext[x=1.187204in,y=4.647354in,,]{\color{textcolor}\fontsize{10.000000}{12.000000}\selectfont \(\displaystyle t\)}%
\end{pgfscope}%
\begin{pgfscope}%
\definecolor{textcolor}{rgb}{0.000000,0.000000,0.000000}%
\pgfsetstrokecolor{textcolor}%
\pgfsetfillcolor{textcolor}%
\pgftext[x=0.339052in,y=5.235466in,,top]{\color{textcolor}\fontsize{10.000000}{12.000000}\selectfont \(\displaystyle 0\)}%
\end{pgfscope}%
\begin{pgfscope}%
\definecolor{textcolor}{rgb}{0.000000,0.000000,0.000000}%
\pgfsetstrokecolor{textcolor}%
\pgfsetfillcolor{textcolor}%
\pgftext[x=0.818850in,y=5.075885in,,top]{\color{textcolor}\fontsize{10.000000}{12.000000}\selectfont \(\displaystyle 2\varepsilon\)}%
\end{pgfscope}%
\begin{pgfscope}%
\definecolor{textcolor}{rgb}{0.000000,0.000000,0.000000}%
\pgfsetstrokecolor{textcolor}%
\pgfsetfillcolor{textcolor}%
\pgftext[x=1.309725in,y=4.912620in,,top]{\color{textcolor}\fontsize{10.000000}{12.000000}\selectfont \(\displaystyle 4\varepsilon\)}%
\end{pgfscope}%
\begin{pgfscope}%
\definecolor{textcolor}{rgb}{0.000000,0.000000,0.000000}%
\pgfsetstrokecolor{textcolor}%
\pgfsetfillcolor{textcolor}%
\pgftext[x=1.812065in,y=4.745541in,,top]{\color{textcolor}\fontsize{10.000000}{12.000000}\selectfont \(\displaystyle 6\varepsilon\)}%
\end{pgfscope}%
\begin{pgfscope}%
\definecolor{textcolor}{rgb}{0.000000,0.000000,0.000000}%
\pgfsetstrokecolor{textcolor}%
\pgfsetfillcolor{textcolor}%
\pgftext[x=2.326278in,y=4.574514in,,top]{\color{textcolor}\fontsize{10.000000}{12.000000}\selectfont \(\displaystyle 8\varepsilon\)}%
\end{pgfscope}%
\begin{pgfscope}%
\definecolor{textcolor}{rgb}{0.000000,0.000000,0.000000}%
\pgfsetstrokecolor{textcolor}%
\pgfsetfillcolor{textcolor}%
\pgftext[x=3.586680in,y=5.014498in,,]{\color{textcolor}\fontsize{10.000000}{12.000000}\selectfont \(\displaystyle x\)}%
\end{pgfscope}%
\begin{pgfscope}%
\definecolor{textcolor}{rgb}{0.000000,0.000000,0.000000}%
\pgfsetstrokecolor{textcolor}%
\pgfsetfillcolor{textcolor}%
\pgftext[x=2.843304in,y=4.654982in,,top]{\color{textcolor}\fontsize{10.000000}{12.000000}\selectfont 0.0}%
\end{pgfscope}%
\begin{pgfscope}%
\definecolor{textcolor}{rgb}{0.000000,0.000000,0.000000}%
\pgfsetstrokecolor{textcolor}%
\pgfsetfillcolor{textcolor}%
\pgftext[x=3.064914in,y=4.888337in,,top]{\color{textcolor}\fontsize{10.000000}{12.000000}\selectfont 0.2}%
\end{pgfscope}%
\begin{pgfscope}%
\definecolor{textcolor}{rgb}{0.000000,0.000000,0.000000}%
\pgfsetstrokecolor{textcolor}%
\pgfsetfillcolor{textcolor}%
\pgftext[x=3.279535in,y=5.114332in,,top]{\color{textcolor}\fontsize{10.000000}{12.000000}\selectfont 0.4}%
\end{pgfscope}%
\begin{pgfscope}%
\definecolor{textcolor}{rgb}{0.000000,0.000000,0.000000}%
\pgfsetstrokecolor{textcolor}%
\pgfsetfillcolor{textcolor}%
\pgftext[x=3.487491in,y=5.333310in,,top]{\color{textcolor}\fontsize{10.000000}{12.000000}\selectfont 0.6}%
\end{pgfscope}%
\begin{pgfscope}%
\definecolor{textcolor}{rgb}{0.000000,0.000000,0.000000}%
\pgfsetstrokecolor{textcolor}%
\pgfsetfillcolor{textcolor}%
\pgftext[x=3.689089in,y=5.545592in,,top]{\color{textcolor}\fontsize{10.000000}{12.000000}\selectfont 0.8}%
\end{pgfscope}%
\begin{pgfscope}%
\definecolor{textcolor}{rgb}{0.000000,0.000000,0.000000}%
\pgfsetstrokecolor{textcolor}%
\pgfsetfillcolor{textcolor}%
\pgftext[x=3.884616in,y=5.751481in,,top]{\color{textcolor}\fontsize{10.000000}{12.000000}\selectfont 1.0}%
\end{pgfscope}%
\begin{pgfscope}%
\definecolor{textcolor}{rgb}{0.000000,0.000000,0.000000}%
\pgfsetstrokecolor{textcolor}%
\pgfsetfillcolor{textcolor}%
\pgftext[x=4.413974in, y=6.698920in, left, base,rotate=87.378092]{\color{textcolor}\fontsize{10.000000}{12.000000}\selectfont \(\displaystyle y_1(\cdot, \cdot)\)}%
\end{pgfscope}%
\begin{pgfscope}%
\definecolor{textcolor}{rgb}{0.000000,0.000000,0.000000}%
\pgfsetstrokecolor{textcolor}%
\pgfsetfillcolor{textcolor}%
\pgftext[x=4.045997in,y=6.091349in,,top]{\color{textcolor}\fontsize{10.000000}{12.000000}\selectfont \ensuremath{-}4}%
\end{pgfscope}%
\begin{pgfscope}%
\definecolor{textcolor}{rgb}{0.000000,0.000000,0.000000}%
\pgfsetstrokecolor{textcolor}%
\pgfsetfillcolor{textcolor}%
\pgftext[x=4.062013in,y=6.396228in,,top]{\color{textcolor}\fontsize{10.000000}{12.000000}\selectfont \ensuremath{-}3}%
\end{pgfscope}%
\begin{pgfscope}%
\definecolor{textcolor}{rgb}{0.000000,0.000000,0.000000}%
\pgfsetstrokecolor{textcolor}%
\pgfsetfillcolor{textcolor}%
\pgftext[x=4.078292in,y=6.706115in,,top]{\color{textcolor}\fontsize{10.000000}{12.000000}\selectfont \ensuremath{-}2}%
\end{pgfscope}%
\begin{pgfscope}%
\definecolor{textcolor}{rgb}{0.000000,0.000000,0.000000}%
\pgfsetstrokecolor{textcolor}%
\pgfsetfillcolor{textcolor}%
\pgftext[x=4.094841in,y=7.021134in,,top]{\color{textcolor}\fontsize{10.000000}{12.000000}\selectfont \ensuremath{-}1}%
\end{pgfscope}%
\begin{pgfscope}%
\definecolor{textcolor}{rgb}{0.000000,0.000000,0.000000}%
\pgfsetstrokecolor{textcolor}%
\pgfsetfillcolor{textcolor}%
\pgftext[x=4.111666in,y=7.341413in,,top]{\color{textcolor}\fontsize{10.000000}{12.000000}\selectfont 0}%
\end{pgfscope}%
\begin{pgfscope}%
\definecolor{textcolor}{rgb}{0.000000,0.000000,0.000000}%
\pgfsetstrokecolor{textcolor}%
\pgfsetfillcolor{textcolor}%
\pgftext[x=5.687204in,y=4.647354in,,]{\color{textcolor}\fontsize{10.000000}{12.000000}\selectfont \(\displaystyle t\)}%
\end{pgfscope}%
\begin{pgfscope}%
\definecolor{textcolor}{rgb}{0.000000,0.000000,0.000000}%
\pgfsetstrokecolor{textcolor}%
\pgfsetfillcolor{textcolor}%
\pgftext[x=4.839052in,y=5.235466in,,top]{\color{textcolor}\fontsize{10.000000}{12.000000}\selectfont 0}%
\end{pgfscope}%
\begin{pgfscope}%
\definecolor{textcolor}{rgb}{0.000000,0.000000,0.000000}%
\pgfsetstrokecolor{textcolor}%
\pgfsetfillcolor{textcolor}%
\pgftext[x=5.222016in,y=5.108092in,,top]{\color{textcolor}\fontsize{10.000000}{12.000000}\selectfont 1}%
\end{pgfscope}%
\begin{pgfscope}%
\definecolor{textcolor}{rgb}{0.000000,0.000000,0.000000}%
\pgfsetstrokecolor{textcolor}%
\pgfsetfillcolor{textcolor}%
\pgftext[x=5.612021in,y=4.978376in,,top]{\color{textcolor}\fontsize{10.000000}{12.000000}\selectfont 2}%
\end{pgfscope}%
\begin{pgfscope}%
\definecolor{textcolor}{rgb}{0.000000,0.000000,0.000000}%
\pgfsetstrokecolor{textcolor}%
\pgfsetfillcolor{textcolor}%
\pgftext[x=6.009263in,y=4.846254in,,top]{\color{textcolor}\fontsize{10.000000}{12.000000}\selectfont 3}%
\end{pgfscope}%
\begin{pgfscope}%
\definecolor{textcolor}{rgb}{0.000000,0.000000,0.000000}%
\pgfsetstrokecolor{textcolor}%
\pgfsetfillcolor{textcolor}%
\pgftext[x=6.413945in,y=4.711656in,,top]{\color{textcolor}\fontsize{10.000000}{12.000000}\selectfont 4}%
\end{pgfscope}%
\begin{pgfscope}%
\definecolor{textcolor}{rgb}{0.000000,0.000000,0.000000}%
\pgfsetstrokecolor{textcolor}%
\pgfsetfillcolor{textcolor}%
\pgftext[x=6.826278in,y=4.574514in,,top]{\color{textcolor}\fontsize{10.000000}{12.000000}\selectfont 5}%
\end{pgfscope}%
\begin{pgfscope}%
\definecolor{textcolor}{rgb}{0.000000,0.000000,0.000000}%
\pgfsetstrokecolor{textcolor}%
\pgfsetfillcolor{textcolor}%
\pgftext[x=8.086680in,y=5.014498in,,]{\color{textcolor}\fontsize{10.000000}{12.000000}\selectfont \(\displaystyle x\)}%
\end{pgfscope}%
\begin{pgfscope}%
\definecolor{textcolor}{rgb}{0.000000,0.000000,0.000000}%
\pgfsetstrokecolor{textcolor}%
\pgfsetfillcolor{textcolor}%
\pgftext[x=7.343304in,y=4.654982in,,top]{\color{textcolor}\fontsize{10.000000}{12.000000}\selectfont 0.0}%
\end{pgfscope}%
\begin{pgfscope}%
\definecolor{textcolor}{rgb}{0.000000,0.000000,0.000000}%
\pgfsetstrokecolor{textcolor}%
\pgfsetfillcolor{textcolor}%
\pgftext[x=7.564914in,y=4.888337in,,top]{\color{textcolor}\fontsize{10.000000}{12.000000}\selectfont 0.2}%
\end{pgfscope}%
\begin{pgfscope}%
\definecolor{textcolor}{rgb}{0.000000,0.000000,0.000000}%
\pgfsetstrokecolor{textcolor}%
\pgfsetfillcolor{textcolor}%
\pgftext[x=7.779535in,y=5.114332in,,top]{\color{textcolor}\fontsize{10.000000}{12.000000}\selectfont 0.4}%
\end{pgfscope}%
\begin{pgfscope}%
\definecolor{textcolor}{rgb}{0.000000,0.000000,0.000000}%
\pgfsetstrokecolor{textcolor}%
\pgfsetfillcolor{textcolor}%
\pgftext[x=7.987491in,y=5.333310in,,top]{\color{textcolor}\fontsize{10.000000}{12.000000}\selectfont 0.6}%
\end{pgfscope}%
\begin{pgfscope}%
\definecolor{textcolor}{rgb}{0.000000,0.000000,0.000000}%
\pgfsetstrokecolor{textcolor}%
\pgfsetfillcolor{textcolor}%
\pgftext[x=8.189089in,y=5.545592in,,top]{\color{textcolor}\fontsize{10.000000}{12.000000}\selectfont 0.8}%
\end{pgfscope}%
\begin{pgfscope}%
\definecolor{textcolor}{rgb}{0.000000,0.000000,0.000000}%
\pgfsetstrokecolor{textcolor}%
\pgfsetfillcolor{textcolor}%
\pgftext[x=8.384616in,y=5.751481in,,top]{\color{textcolor}\fontsize{10.000000}{12.000000}\selectfont 1.0}%
\end{pgfscope}%
\begin{pgfscope}%
\definecolor{textcolor}{rgb}{0.000000,0.000000,0.000000}%
\pgfsetstrokecolor{textcolor}%
\pgfsetfillcolor{textcolor}%
\pgftext[x=8.913974in, y=6.698920in, left, base,rotate=87.378092]{\color{textcolor}\fontsize{10.000000}{12.000000}\selectfont \(\displaystyle y_1(\cdot, \cdot)\)}%
\end{pgfscope}%
\begin{pgfscope}%
\definecolor{textcolor}{rgb}{0.000000,0.000000,0.000000}%
\pgfsetstrokecolor{textcolor}%
\pgfsetfillcolor{textcolor}%
\pgftext[x=8.545997in,y=6.091349in,,top]{\color{textcolor}\fontsize{10.000000}{12.000000}\selectfont \ensuremath{-}4}%
\end{pgfscope}%
\begin{pgfscope}%
\definecolor{textcolor}{rgb}{0.000000,0.000000,0.000000}%
\pgfsetstrokecolor{textcolor}%
\pgfsetfillcolor{textcolor}%
\pgftext[x=8.562013in,y=6.396228in,,top]{\color{textcolor}\fontsize{10.000000}{12.000000}\selectfont \ensuremath{-}3}%
\end{pgfscope}%
\begin{pgfscope}%
\definecolor{textcolor}{rgb}{0.000000,0.000000,0.000000}%
\pgfsetstrokecolor{textcolor}%
\pgfsetfillcolor{textcolor}%
\pgftext[x=8.578292in,y=6.706115in,,top]{\color{textcolor}\fontsize{10.000000}{12.000000}\selectfont \ensuremath{-}2}%
\end{pgfscope}%
\begin{pgfscope}%
\definecolor{textcolor}{rgb}{0.000000,0.000000,0.000000}%
\pgfsetstrokecolor{textcolor}%
\pgfsetfillcolor{textcolor}%
\pgftext[x=8.594841in,y=7.021134in,,top]{\color{textcolor}\fontsize{10.000000}{12.000000}\selectfont \ensuremath{-}1}%
\end{pgfscope}%
\begin{pgfscope}%
\definecolor{textcolor}{rgb}{0.000000,0.000000,0.000000}%
\pgfsetstrokecolor{textcolor}%
\pgfsetfillcolor{textcolor}%
\pgftext[x=8.611666in,y=7.341413in,,top]{\color{textcolor}\fontsize{10.000000}{12.000000}\selectfont 0}%
\end{pgfscope}%
\begin{pgfscope}%
\definecolor{textcolor}{rgb}{0.000000,0.000000,0.000000}%
\pgfsetstrokecolor{textcolor}%
\pgfsetfillcolor{textcolor}%
\pgftext[x=1.187204in,y=0.417354in,,]{\color{textcolor}\fontsize{10.000000}{12.000000}\selectfont \(\displaystyle t\)}%
\end{pgfscope}%
\begin{pgfscope}%
\definecolor{textcolor}{rgb}{0.000000,0.000000,0.000000}%
\pgfsetstrokecolor{textcolor}%
\pgfsetfillcolor{textcolor}%
\pgftext[x=0.339052in,y=1.005466in,,top]{\color{textcolor}\fontsize{10.000000}{12.000000}\selectfont \(\displaystyle 0\)}%
\end{pgfscope}%
\begin{pgfscope}%
\definecolor{textcolor}{rgb}{0.000000,0.000000,0.000000}%
\pgfsetstrokecolor{textcolor}%
\pgfsetfillcolor{textcolor}%
\pgftext[x=0.818850in,y=0.845885in,,top]{\color{textcolor}\fontsize{10.000000}{12.000000}\selectfont \(\displaystyle 2\varepsilon\)}%
\end{pgfscope}%
\begin{pgfscope}%
\definecolor{textcolor}{rgb}{0.000000,0.000000,0.000000}%
\pgfsetstrokecolor{textcolor}%
\pgfsetfillcolor{textcolor}%
\pgftext[x=1.309725in,y=0.682620in,,top]{\color{textcolor}\fontsize{10.000000}{12.000000}\selectfont \(\displaystyle 4\varepsilon\)}%
\end{pgfscope}%
\begin{pgfscope}%
\definecolor{textcolor}{rgb}{0.000000,0.000000,0.000000}%
\pgfsetstrokecolor{textcolor}%
\pgfsetfillcolor{textcolor}%
\pgftext[x=1.812065in,y=0.515541in,,top]{\color{textcolor}\fontsize{10.000000}{12.000000}\selectfont \(\displaystyle 6\varepsilon\)}%
\end{pgfscope}%
\begin{pgfscope}%
\definecolor{textcolor}{rgb}{0.000000,0.000000,0.000000}%
\pgfsetstrokecolor{textcolor}%
\pgfsetfillcolor{textcolor}%
\pgftext[x=2.326278in,y=0.344514in,,top]{\color{textcolor}\fontsize{10.000000}{12.000000}\selectfont \(\displaystyle 8\varepsilon\)}%
\end{pgfscope}%
\begin{pgfscope}%
\definecolor{textcolor}{rgb}{0.000000,0.000000,0.000000}%
\pgfsetstrokecolor{textcolor}%
\pgfsetfillcolor{textcolor}%
\pgftext[x=3.586680in,y=0.784498in,,]{\color{textcolor}\fontsize{10.000000}{12.000000}\selectfont \(\displaystyle x\)}%
\end{pgfscope}%
\begin{pgfscope}%
\definecolor{textcolor}{rgb}{0.000000,0.000000,0.000000}%
\pgfsetstrokecolor{textcolor}%
\pgfsetfillcolor{textcolor}%
\pgftext[x=2.843304in,y=0.424982in,,top]{\color{textcolor}\fontsize{10.000000}{12.000000}\selectfont 0.0}%
\end{pgfscope}%
\begin{pgfscope}%
\definecolor{textcolor}{rgb}{0.000000,0.000000,0.000000}%
\pgfsetstrokecolor{textcolor}%
\pgfsetfillcolor{textcolor}%
\pgftext[x=3.064914in,y=0.658337in,,top]{\color{textcolor}\fontsize{10.000000}{12.000000}\selectfont 0.2}%
\end{pgfscope}%
\begin{pgfscope}%
\definecolor{textcolor}{rgb}{0.000000,0.000000,0.000000}%
\pgfsetstrokecolor{textcolor}%
\pgfsetfillcolor{textcolor}%
\pgftext[x=3.279535in,y=0.884332in,,top]{\color{textcolor}\fontsize{10.000000}{12.000000}\selectfont 0.4}%
\end{pgfscope}%
\begin{pgfscope}%
\definecolor{textcolor}{rgb}{0.000000,0.000000,0.000000}%
\pgfsetstrokecolor{textcolor}%
\pgfsetfillcolor{textcolor}%
\pgftext[x=3.487491in,y=1.103310in,,top]{\color{textcolor}\fontsize{10.000000}{12.000000}\selectfont 0.6}%
\end{pgfscope}%
\begin{pgfscope}%
\definecolor{textcolor}{rgb}{0.000000,0.000000,0.000000}%
\pgfsetstrokecolor{textcolor}%
\pgfsetfillcolor{textcolor}%
\pgftext[x=3.689089in,y=1.315592in,,top]{\color{textcolor}\fontsize{10.000000}{12.000000}\selectfont 0.8}%
\end{pgfscope}%
\begin{pgfscope}%
\definecolor{textcolor}{rgb}{0.000000,0.000000,0.000000}%
\pgfsetstrokecolor{textcolor}%
\pgfsetfillcolor{textcolor}%
\pgftext[x=3.884616in,y=1.521481in,,top]{\color{textcolor}\fontsize{10.000000}{12.000000}\selectfont 1.0}%
\end{pgfscope}%
\begin{pgfscope}%
\definecolor{textcolor}{rgb}{0.000000,0.000000,0.000000}%
\pgfsetstrokecolor{textcolor}%
\pgfsetfillcolor{textcolor}%
\pgftext[x=4.413974in, y=2.468920in, left, base,rotate=87.378092]{\color{textcolor}\fontsize{10.000000}{12.000000}\selectfont \(\displaystyle y_2(\cdot, \cdot)\)}%
\end{pgfscope}%
\begin{pgfscope}%
\definecolor{textcolor}{rgb}{0.000000,0.000000,0.000000}%
\pgfsetstrokecolor{textcolor}%
\pgfsetfillcolor{textcolor}%
\pgftext[x=4.058026in,y=2.090333in,,top]{\color{textcolor}\fontsize{10.000000}{12.000000}\selectfont \ensuremath{-}0.6}%
\end{pgfscope}%
\begin{pgfscope}%
\definecolor{textcolor}{rgb}{0.000000,0.000000,0.000000}%
\pgfsetstrokecolor{textcolor}%
\pgfsetfillcolor{textcolor}%
\pgftext[x=4.074386in,y=2.401766in,,top]{\color{textcolor}\fontsize{10.000000}{12.000000}\selectfont \ensuremath{-}0.4}%
\end{pgfscope}%
\begin{pgfscope}%
\definecolor{textcolor}{rgb}{0.000000,0.000000,0.000000}%
\pgfsetstrokecolor{textcolor}%
\pgfsetfillcolor{textcolor}%
\pgftext[x=4.091019in,y=2.718392in,,top]{\color{textcolor}\fontsize{10.000000}{12.000000}\selectfont \ensuremath{-}0.2}%
\end{pgfscope}%
\begin{pgfscope}%
\definecolor{textcolor}{rgb}{0.000000,0.000000,0.000000}%
\pgfsetstrokecolor{textcolor}%
\pgfsetfillcolor{textcolor}%
\pgftext[x=4.107932in,y=3.040344in,,top]{\color{textcolor}\fontsize{10.000000}{12.000000}\selectfont 0.0}%
\end{pgfscope}%
\begin{pgfscope}%
\definecolor{textcolor}{rgb}{0.000000,0.000000,0.000000}%
\pgfsetstrokecolor{textcolor}%
\pgfsetfillcolor{textcolor}%
\pgftext[x=4.125132in,y=3.367756in,,top]{\color{textcolor}\fontsize{10.000000}{12.000000}\selectfont 0.2}%
\end{pgfscope}%
\begin{pgfscope}%
\definecolor{textcolor}{rgb}{0.000000,0.000000,0.000000}%
\pgfsetstrokecolor{textcolor}%
\pgfsetfillcolor{textcolor}%
\pgftext[x=5.687204in,y=0.417354in,,]{\color{textcolor}\fontsize{10.000000}{12.000000}\selectfont \(\displaystyle t\)}%
\end{pgfscope}%
\begin{pgfscope}%
\definecolor{textcolor}{rgb}{0.000000,0.000000,0.000000}%
\pgfsetstrokecolor{textcolor}%
\pgfsetfillcolor{textcolor}%
\pgftext[x=4.839052in,y=1.005466in,,top]{\color{textcolor}\fontsize{10.000000}{12.000000}\selectfont 0}%
\end{pgfscope}%
\begin{pgfscope}%
\definecolor{textcolor}{rgb}{0.000000,0.000000,0.000000}%
\pgfsetstrokecolor{textcolor}%
\pgfsetfillcolor{textcolor}%
\pgftext[x=5.222016in,y=0.878092in,,top]{\color{textcolor}\fontsize{10.000000}{12.000000}\selectfont 1}%
\end{pgfscope}%
\begin{pgfscope}%
\definecolor{textcolor}{rgb}{0.000000,0.000000,0.000000}%
\pgfsetstrokecolor{textcolor}%
\pgfsetfillcolor{textcolor}%
\pgftext[x=5.612021in,y=0.748376in,,top]{\color{textcolor}\fontsize{10.000000}{12.000000}\selectfont 2}%
\end{pgfscope}%
\begin{pgfscope}%
\definecolor{textcolor}{rgb}{0.000000,0.000000,0.000000}%
\pgfsetstrokecolor{textcolor}%
\pgfsetfillcolor{textcolor}%
\pgftext[x=6.009263in,y=0.616254in,,top]{\color{textcolor}\fontsize{10.000000}{12.000000}\selectfont 3}%
\end{pgfscope}%
\begin{pgfscope}%
\definecolor{textcolor}{rgb}{0.000000,0.000000,0.000000}%
\pgfsetstrokecolor{textcolor}%
\pgfsetfillcolor{textcolor}%
\pgftext[x=6.413945in,y=0.481656in,,top]{\color{textcolor}\fontsize{10.000000}{12.000000}\selectfont 4}%
\end{pgfscope}%
\begin{pgfscope}%
\definecolor{textcolor}{rgb}{0.000000,0.000000,0.000000}%
\pgfsetstrokecolor{textcolor}%
\pgfsetfillcolor{textcolor}%
\pgftext[x=6.826278in,y=0.344514in,,top]{\color{textcolor}\fontsize{10.000000}{12.000000}\selectfont 5}%
\end{pgfscope}%
\begin{pgfscope}%
\definecolor{textcolor}{rgb}{0.000000,0.000000,0.000000}%
\pgfsetstrokecolor{textcolor}%
\pgfsetfillcolor{textcolor}%
\pgftext[x=8.086680in,y=0.784498in,,]{\color{textcolor}\fontsize{10.000000}{12.000000}\selectfont \(\displaystyle x\)}%
\end{pgfscope}%
\begin{pgfscope}%
\definecolor{textcolor}{rgb}{0.000000,0.000000,0.000000}%
\pgfsetstrokecolor{textcolor}%
\pgfsetfillcolor{textcolor}%
\pgftext[x=7.343304in,y=0.424982in,,top]{\color{textcolor}\fontsize{10.000000}{12.000000}\selectfont 0.0}%
\end{pgfscope}%
\begin{pgfscope}%
\definecolor{textcolor}{rgb}{0.000000,0.000000,0.000000}%
\pgfsetstrokecolor{textcolor}%
\pgfsetfillcolor{textcolor}%
\pgftext[x=7.564914in,y=0.658337in,,top]{\color{textcolor}\fontsize{10.000000}{12.000000}\selectfont 0.2}%
\end{pgfscope}%
\begin{pgfscope}%
\definecolor{textcolor}{rgb}{0.000000,0.000000,0.000000}%
\pgfsetstrokecolor{textcolor}%
\pgfsetfillcolor{textcolor}%
\pgftext[x=7.779535in,y=0.884332in,,top]{\color{textcolor}\fontsize{10.000000}{12.000000}\selectfont 0.4}%
\end{pgfscope}%
\begin{pgfscope}%
\definecolor{textcolor}{rgb}{0.000000,0.000000,0.000000}%
\pgfsetstrokecolor{textcolor}%
\pgfsetfillcolor{textcolor}%
\pgftext[x=7.987491in,y=1.103310in,,top]{\color{textcolor}\fontsize{10.000000}{12.000000}\selectfont 0.6}%
\end{pgfscope}%
\begin{pgfscope}%
\definecolor{textcolor}{rgb}{0.000000,0.000000,0.000000}%
\pgfsetstrokecolor{textcolor}%
\pgfsetfillcolor{textcolor}%
\pgftext[x=8.189089in,y=1.315592in,,top]{\color{textcolor}\fontsize{10.000000}{12.000000}\selectfont 0.8}%
\end{pgfscope}%
\begin{pgfscope}%
\definecolor{textcolor}{rgb}{0.000000,0.000000,0.000000}%
\pgfsetstrokecolor{textcolor}%
\pgfsetfillcolor{textcolor}%
\pgftext[x=8.384616in,y=1.521481in,,top]{\color{textcolor}\fontsize{10.000000}{12.000000}\selectfont 1.0}%
\end{pgfscope}%
\begin{pgfscope}%
\definecolor{textcolor}{rgb}{0.000000,0.000000,0.000000}%
\pgfsetstrokecolor{textcolor}%
\pgfsetfillcolor{textcolor}%
\pgftext[x=8.913974in, y=2.468920in, left, base,rotate=87.378092]{\color{textcolor}\fontsize{10.000000}{12.000000}\selectfont \(\displaystyle y_2(\cdot, \cdot)\)}%
\end{pgfscope}%
\begin{pgfscope}%
\definecolor{textcolor}{rgb}{0.000000,0.000000,0.000000}%
\pgfsetstrokecolor{textcolor}%
\pgfsetfillcolor{textcolor}%
\pgftext[x=8.558026in,y=2.090333in,,top]{\color{textcolor}\fontsize{10.000000}{12.000000}\selectfont \ensuremath{-}0.6}%
\end{pgfscope}%
\begin{pgfscope}%
\definecolor{textcolor}{rgb}{0.000000,0.000000,0.000000}%
\pgfsetstrokecolor{textcolor}%
\pgfsetfillcolor{textcolor}%
\pgftext[x=8.574386in,y=2.401766in,,top]{\color{textcolor}\fontsize{10.000000}{12.000000}\selectfont \ensuremath{-}0.4}%
\end{pgfscope}%
\begin{pgfscope}%
\definecolor{textcolor}{rgb}{0.000000,0.000000,0.000000}%
\pgfsetstrokecolor{textcolor}%
\pgfsetfillcolor{textcolor}%
\pgftext[x=8.591019in,y=2.718392in,,top]{\color{textcolor}\fontsize{10.000000}{12.000000}\selectfont \ensuremath{-}0.2}%
\end{pgfscope}%
\begin{pgfscope}%
\definecolor{textcolor}{rgb}{0.000000,0.000000,0.000000}%
\pgfsetstrokecolor{textcolor}%
\pgfsetfillcolor{textcolor}%
\pgftext[x=8.607932in,y=3.040344in,,top]{\color{textcolor}\fontsize{10.000000}{12.000000}\selectfont 0.0}%
\end{pgfscope}%
\begin{pgfscope}%
\definecolor{textcolor}{rgb}{0.000000,0.000000,0.000000}%
\pgfsetstrokecolor{textcolor}%
\pgfsetfillcolor{textcolor}%
\pgftext[x=8.625132in,y=3.367756in,,top]{\color{textcolor}\fontsize{10.000000}{12.000000}\selectfont 0.2}%
\end{pgfscope}%
\begin{pgfscope}%
\definecolor{textcolor}{rgb}{0.000000,0.000000,0.000000}%
\pgfsetstrokecolor{textcolor}%
\pgfsetfillcolor{textcolor}%
\pgftext[x=4.500000in,y=8.820000in,,top]{\color{textcolor}\fontsize{12.000000}{14.400000}\selectfont Solution of the PDE with \(\displaystyle \varepsilon = 0.01\)}%
\end{pgfscope}%
\end{pgfpicture}%
\makeatother%
\endgroup%